\documentclass{amsart}
\usepackage{amsmath}
\usepackage{amsthm}
\usepackage{amssymb}
\usepackage{xspace}
\usepackage{enumitem}

\theoremstyle{plain}
\newtheorem{theorem}{Theorem}[section]
\newtheorem{proposition}[theorem]{Proposition}
\newtheorem{lemma}[theorem]{Lemma}
\newtheorem*{sublemma}{Sublemma}

\theoremstyle{definition}

\newcommand{\comment}[1]{}

\newcommand{\rea}{\mathbf{R}}
\newcommand{\nat}{\mathbf{N}}

\renewcommand{\P}{\ensuremath{\mathbf{P}}\xspace}
\newcommand{\intd}[1]{\,\mathrm{d}#1 \,}
\newcommand{\floor}[1]{\left\lfloor #1 \right\rfloor}
\newcommand{\ceil}[1]{\left\lceil #1 \right\rceil}

\newcommand{\hmeas}{\ensuremath{\mathcal{H}}\xspace}

\newcommand{\overbar}[1]{{\mkern 3mu\overline{\mkern-3mu#1\mkern-1mu}\mkern 1mu}}
\newcommand{\oball}[2]{\ensuremath{B{\left(#1, #2\right)}}\xspace}
\newcommand{\cball}[2]{\ensuremath{\overbar{B}{\left(#1, #2\right)}}\xspace}
\newcommand{\crect}[2]{\ensuremath{\overbar R \left(#1, #2\right)}\xspace}
\newcommand{\probs}{\ensuremath{\mathcal{P}}\xspace}
\renewcommand{\Cap}{\Capp}

\newcommand{\heis}{\ensuremath{\mathbf{H}}\xspace}
\newcommand{\dheis}{\ensuremath{d_\heis}\xspace}

\newcommand{\seq}[1]{\ensuremath{\underline{#1}}\xspace}

\newcommand\restr[2]{{
  \left.\kern-\nulldelimiterspace
  #1
  \vphantom{|}
  \right|_{#2}
  }}

\DeclareMathOperator{\supp}{supp}

\DeclareMathOperator{\dimh}{dim_H}

\DeclareMathOperator{\Capp}{Cap}
\DeclareMathOperator{\E}{\mathbf{E}}
\DeclareMathOperator{\Var}{\mathbf{Var}}

\begin{document}

\title[Limsup sets in the Heisenberg group]{Hausdorff dimension of limsup sets
  of rectangles in the Heisenberg group}

\author[F. Ekstr\"om]{Fredrik Ekstr\"om$^1$}
\address{Department of Mathematical Sciences, P.O. Box 3000,
         90014 University of Oulu, Finland$^{1,2,3}$}
\email{j.fredrik.e@gmail.com$^1$}

\author[E. J\"arvenp\"a\"a]{Esa J\"arvenp\"a\"a$^2$}
\email{esa.jarvenpaa@oulu.fi$^2$}

\author[M. J\"arvenp\"a\"a]{Maarit J\"arvenp\"a\"a$^3$}
\email{maarit.jarvenpaa@oulu.fi$^3$}


\begin{abstract}
The almost sure value of the Hausdorff dimension of limsup sets generated
by randomly distributed rect\-angles in the Heisenberg group is computed in
terms of directed singular value functions.
\end{abstract}

\maketitle

\section{Introduction}

Dimensional properties of subsets of Heisenberg groups have attained a lot
of interest recently. Due to the non-trivial relation between the Hausdorff
dimensions with respect to the Euclidean and the Heisenberg
metrics \cite{BaloghRicklyCassano03, BaloghTyson05}, one cannot directly
transfer dimensional results in Euclidean spaces into Heisenberg groups.
Indeed, it turns out that some theorems concerning dimensions have a special
flavour or even an essentially different form in the Heisenberg setting.
These include, for example, dimensional properties of self-affine sets,
projections and slices.

In the Heisenberg group Hausdorff dimensions of self-similar and self-affine
sets have been studied e.g.~in
\cite{BBMT10,BaloghTyson05,BaloghTysonWarhurst09,ChousionisTysonUrbanski19}.
Even though the class of affine iterated function systems is quite
restrictive\,---\,every such system is a horizontal lift of an affine iterated
function system on the plane\,---\,the dimension calculations involve some
subtleties. The behaviour of the Hausdorff dimension under
projections and slicing transpires to be interesting, see
\cite{BD-CFMT13,BFMT12,FasslerHovila16,Hovila14}. There are two kinds of natural
projections (and slices) in Heisenberg groups\,---\,the horizontal
and vertical ones. The vertical projections possess an exceptional
feature:~they are not Lipschitz continuous. This indicates that the
methods developed in the Euclidean setting cannot be utilised. For
related questions concerning Sobolev maps and the foliations generated by
the horizontal subspaces, see
\cite{BaloghTysonWildrick17}.

In this paper, we initiate a new direction of research in Heisenberg groups by
investigating dimensions of limsup sets generated by rectangles.
Let $X$ be a space and let $(A_n)$ be a sequence of subsets of $X$.
The limsup set generated by the sequence $(A_n)$ consists of those points of
$X$ which are covered by infinitely many of the sets $A_n$, that is, 
        \[
        \limsup_n A_n = \bigcap_{n=1}^\infty \bigcup_{k=n}^\infty A_k.
        \]
Limsup sets are encountered in many fields of
mathematics\,---\,one of the earliest appearances being the Borel-Cantelli
lemma \cite{Borel97,Cantelli17}. They play a central role in the study of
Besicovitch-Eggleston sets concerning the $k$-adic expansions of real numbers
\cite{Besicovitch34,Eggleston49} as well as in Diophantine
approximation \cite{Jarnik32,Khintchine26}. For more information on different
aspects of limsup sets, we refer to \cite{FJJSUP} and the references therein.
For recent results regarding Diophantine approximation in Heisenberg
groups, see \cite{SeuretVigneron17,Vandehey16,ZhengUP}. 

Dimensional properties of random limsup sets have been actively studied, see
for example \cite{Durand10,EkstromPerssonUP,FanSchmelingTroubetzkoy13,
  FanWu04,FJJSUP,JJKLS14,JJKLSX17,Persson15,PerssonUP,SeuretUP}. Combining the
results of these papers, the almost sure value of dimension of random limsup
sets is known in the following cases:
\begin{itemize}
\item[-]
the underlying space $X$ is a Riemann manifold, the generating sets
$(A_n)$ are Lebesgue measurable with positive density and the driving measure
determining the randomness is not singular with respect to the Lebesgue measure,

\item[-]
$X$ is the Euclidean or the symbolic space, the generating sets $(A_n)$ are
balls and the driving measure has special properties like being a Gibbs measure,

\item[-]
$X$ is an Ahlfors regular
metric space, randomness is given by the natural measure and $(A_n)$ is
a sequence of balls.
\end{itemize}

In \cite{EJJS18} a dimension formula for limsup sets generated by rectangles in
products of Ahlfors regular metric spaces is derived. In this paper, we
address the problem of determining the Hausdorff dimension of
random limsup sets generated by rectangles in the first Heisenberg group 
(see Theorem~\ref{mainthm} below). In \cite{EJJS18} the Lipschitz
continuity of projections is utilised to a great extent, and because of that,
the same methods cannot be used in our setting. Instead, we will extend
some results known in the Euclidean setting to unimodular groups or to
compact metric spaces, and make calculations specific to the Heisenberg
group to complete the argument.

We proceed by introducing our notation.
The \emph{Heisenberg group} \heis is the set $\rea^3$ with the non-commutative
group operation
	\[
	pp' = (x + x', y + y', z + z' + 2(x y' - y x')),
	\]
where $p = (x, y, z)$ and $p' = (x', y', z')$. The unit in \heis is $(0, 0, 0)$
and the inverse of $p$ is $p^{-1} = (-x, -y, -z)$. There is a norm on
\heis given by
	\[
	\| p \| = \left(\left(x^2 + y^2\right)^2 + z^2 \right)^{1 / 4},
	\]
which gives rise to a left-invariant metric
	\[
	\dheis(p, p') = \| p^{-1} p' \|
	= \left(\left((x' - x)^2 + (y' - y)^2\right)^2 +
	(z' - z - 2(xy' - yx'))^2 \right)^{1 / 4}.
	\]

Both left and right translation in \heis move vertical lines to
vertical lines in such a way that the Euclidean distance between lines is
preserved, and the image of the Lebesgue measure on a vertical
line under translation is the Lebesgue measure on the image line.
Thus Fubini's theorem implies that the Lebesgue measure on $\rea^3$
is invariant under translations in \heis. It is easy to see that the
Lebesgue measure of $\oball 0 r$ in \heis is proportional
to $r^4$, and by translation invariance the same is true
for every ball of radius $r$. In particular, the metric space
$(\heis, \dheis)$ has Hausdorff dimension $4$.

Let 
        \[
        L(0) = \left\{p'; \, x' = y' = 0 \right\} \text{ and }
        H(0) = \left\{p'; \, z' = 0 \right\}
        \]
be the vertical line and horizontal plane through the origin, respectively.
For $p\in\heis$, define
        \[
        L(p) = pL(0)\text{ and }H(p) = pH(0).
        \]
Then $L(p)$ is the vertical line through $p$ and
	\[
	H(p) = \left\{p'; \, z' = z + 2(xy' - yx') \right\}
	\]
is the plane through $p$ that has slope $0$ in the direction $(x, y)$
and slope \mbox{$2(x^2 + y^2)^{1 / 2}$} in the orthogonal direction $(-y, x)$.
Note that
	\[
	\dheis(p, p') \geq
	\left((x' - x)^2 + (y' - y)^2\right)^{1 / 2},
	\]
with equality if and only if $p' \in H(p)$. It follows that the
distance to $L(p')$ from any point on $L(p)$ equals
the Euclidean distance from $(x, y)$ to $(x', y')$, and vice versa
by symmetry. Thus vertical lines are parallel in the Heisenberg metric,
and the distance between them is the same as the Euclidean distance.
The symmetry of the metric implies that $p' \in H(p)$ if and only if
$p \in H(p')$ and, by definition, 
$pH(p') = H(pp')$. If $p' \in L(p)$ then $H(p')$ is parallel in the
Euclidean sense to $H(p)$.

A closed \emph{rectangle} in \heis \emph{centred at $0$} is
        \[
	\crect 0 r =
	\left\{
	p'; \, x'^2 + y'^2 \leq r_1^2, \, |z'| \leq r_2^2
	\right\},
	\]
and, in general, a closed \emph{rectangle centred at $p$} is a set of the form
        \[
        \crect p r = p \crect 0 r =
        \left\{
	p'; \, (x' - x)^2 + (y' - y)^2 \leq r_1^2, \,
	|z' - z - 2(xy' - yx')| \leq r_2^2
	\right\},
        \]
where $r = (r_1, r_2)$.
This is the set of points that can be reached
from $p$ by moving ``horizontally'' in $H(p)$ a distance at most $r_1$
and then vertically a distance at most $r_2$, or by moving first
vertically a distance at most $r_2$ and then horizontally a distance
at most $r_1$.

Let $\seq p = (p_n)$ be a sequence of points in \heis and let $\seq r = (r_n)$
be a bounded sequence of pairs of positive numbers, and define
	\[
	E_{\seq r}(\seq p) = \limsup_n \crect{p_n}{r_n}.
	\]
The purpose of this article is to give a formula for the Hausdorff dimension
of such a set when the centres of the rectangles are chosen randomly.
Let $\lambda$ be the Lebesgue measure on \heis and let $W$ be a bounded open
subset of \heis. Let $\lambda_W = \lambda(W)^{-1} \restr\lambda W$ and define
the probability space $(\Omega, \P)$ by $\Omega = \heis^\nat$ and
$\P = \lambda_W^\nat$.
Then $\omega \mapsto E_{\seq r}(\omega)$ can be considered as a random set
defined on $(\Omega, \P)$.
The \emph{directed singular value function} is defined as follows:~for
$r = (r_1, r_2)$, if $r_1 \leq r_2$ let
	\[
	\Phi^t(r) =
	\begin{cases}
	r_2^t & \text{if } t \in [0, 2],\\
	r_1^{t - 2} r_2^2& \text{if } t \in [2, 4],
	\end{cases}
	\]
and if $r_1 \geq r_2$ let
	\[
	\Phi^t(r) =
	\begin{cases}
	r_1^t & \text{if } t \in [0, 3], \\
	r_1^{6 - t} r_2^{2(t - 3)} & \text{if } t \in [3, 4].
	\end{cases}
	\]
The minimum of numbers $a,b\in\rea$ is denoted by $a\wedge b$.
        
\begin{theorem}	\label{mainthm}
With the above notation,
	\[
	\dimh E_{\seq r} = \inf\left\{
	t; \, \sum_n \Phi^t(r_n) < \infty
	\right\} \wedge 4
	\]
$\P$-almost surely.        
\end{theorem}

Let $X$ be a metric space. The $t$-dimensional \emph{Hausdorff content}
of a subset $A$ of $X$ is defined by
	\[
	\hmeas_\infty^t(A) = \left\{
	\sum_n |C_n|^t; \, \{C_n\} \text{ is a cover of $A$}
	\right\},
	\]
where the diameter of a set $C$ is denoted by $|C|$.
If $\mu$ is a Borel measure on $X$ and $t > 0$, the \emph{$t$-energy} of
$\mu$ is defined by
	\[
	I_t(\mu) = \iint d(x, y)^{-t} \intd{\mu(y)} \intd{\mu(x)}.
	\]
The \emph{$t$-capacity} of a Borel subset $A$ of $X$ is defined by
	\[
	\Cap_t(A) = \sup\left\{
	\frac{1}{I_t(\mu)}; \, \mu \in \probs(A)
	\right\},
	\]
where $\probs(A)$ denotes the set of Borel probability measures on $X$
that give full measure to $A$. It can be shown that
\begin{equation}\label{caplesscontent}
  \Cap_t(A) \leq \hmeas_\infty^t(A)
\end{equation}
always holds (see \cite[Remark 3.3]{FJJSUP}).

The proof of Theorem~\ref{mainthm} is based on estimating the
Hausdorff content and capacity of rectangles in the Heisenberg
group. The following proposition is proved in
Section~\ref{contentcapacitysection}.

\begin{proposition}\label{phiestimate}
For all $t\in (0,4)\setminus\{1,2,3\}$, there is a constant $c_t$ such that
	\[
	c_t^{-1} \hmeas_\infty^t\left(\crect x r\right)
	\leq \Phi^t(r) \leq
	c_t \Cap_t\left(\crect x r \right)
	\]
for every $x$ and $r$.
\end{proposition}

This immediately implies that if $\sum_n \Phi^t(r_n) < \infty$
then $\hmeas^t(E_{\seq r}(\omega)) = 0$ for every $\omega$, since
every tail of the sequence of rectangles is a cover of $E_{\seq r}(\omega)$.
The almost sure lower bound for $\dimh E_{\seq r}$ then follows
from the estimate of the capacity of a rectangle together with
Theorem~\ref{auxthm} below.

A Borel measure $\mu$ on a metric space $X$ is \emph{$(c, d)$-regular}
if
	\[
	c^{-1} r^d \leq \mu\left(\oball x r\right) \leq c r^d
	\]
for every $x \in X$ and $r \in [0, |X|]$. A locally compact group $G$ is
\emph{unimodular} if the left invariant Haar measure is also right invariant,
or equivalently, if it is invariant under inversion. For example, the
Heisenberg group is $(c, 4)$-regular for some $c$ and unimodular.

Theorem~\ref{auxthm} is a variant of \cite[Theorem 1.1(b)]{FJJSUP} and
\cite[Theorem 1]{Persson15}. We give a somewhat different proof even though
the main philosophy follows the same lines as the proofs in
\cite{FJJSUP,Persson15}.

\begin{theorem}	\label{auxthm}
Let\/ $G$ be a unimodular group with left invariant metric and $(c, d)$-regular
Haar measure $\lambda$, and let\/ $W$ be a bounded open subset of $G$ such that
$\lambda(W) = 1$. Define the probability space $(\Omega, \P)$
by\/ $\Omega = G^\nat$ and\/ $\P = (\restr\lambda W)^\nat$.
Let\/ $(V_n)$ be a bounded sequence of open subsets of\/ $G$ (\emph{bounded}
meaning that there is a ball in $G$ that contains\/ $V_n$ for every $n$). For
$\omega = (\omega_n) \in \Omega$ let
	\[
	E(\omega) = \limsup_n \left(\omega_n V_n\right).
	\]
If\/ $t \in (0, d)$ and
	\[
	\sum_n \Cap_t(V_n) = \infty
	\]
then almost every $\omega$ is such that
$\hmeas^t\left(\mathcal U \cap E(\omega)\right) = \infty$
for every open subset $\mathcal U$ of\/ $W$.
\end{theorem}

The assumption that $\lambda(W) = 1$ makes the proof slightly easier to
read, but is not essential since a constant multiple of a $(c, d)$-regular
Haar measure is a $(c', d)$-regular Haar measure for some $c'$.

The proof of Theorem~\ref{auxthm} is based on the following deterministic
lemma (compare \cite[Lemma 7]{Falconer94}, \cite[Proposition 4.6]{FJJSUP}
and \cite[Theorem 1]{PerssonReeve15}), which is proved in
Section~\ref{proofofmainlemma}. If $X$ is a locally compact metric space,
the \emph{weak-$*$ topology} on the space of Radon measures on $X$
is the topology generated by the maps $\{\mu \mapsto \mu(\varphi) \}$, where
$\varphi$ ranges over the continuous compactly supported functions
$X \to \rea$. It is not
difficult to see that $\lim_{n \to \infty} \mu_n = \mu$ in the weak-$*$
topology if and only if $\lim_{n \to \infty} \mu_n(\varphi) = \mu(\varphi)$
for every continuous compactly supported function $\varphi$.
Naturally, this also defines the weak-$*$ topology in the case $X$ is compact.

\begin{lemma} \label{mainlemma}
Let\/ $\nu$ be a finite Borel measure on a compact metric space
$X$ and let\/ $(\varphi_n)_{n = 1}^\infty$ be a sequence of non-negative
continuous functions on\/ $X$ such that\/
$\lim_{n \to \infty} \varphi_n\intd\nu = \nu$ in the weak-$*$ topology
and\/ $\liminf_{n \to \infty} I_t(\varphi_n\rho\intd\nu) \leq I_t(\rho\intd\nu)$
whenever $\rho$ is a product of finitely many of the
functions\/ $\{\varphi_n\}$. Then for every $t > 0$,
	\[
	\Cap_t\left(\supp \nu \cap \limsup_n (\supp \varphi_n)\right)
	\geq \frac{\nu(X)^2}{I_t(\nu)}.
	\]
\end{lemma}

\subsection*{Notation and conventions}
All measures appearing below are Radon measures, but this will not be
explicitly stated. Thus ``measure'' below means ``Radon measure''.
If $X$ is a metric space with a measure $\mu$ and $\varphi$ is a non-negative
continuous function on $X$, then $I_t(\varphi)$ means $I_t(\varphi\intd\mu)$.
Similarly, if $A$ is a Borel subset of $X$ then $I_t(A)$ means
$I_t(\restr \mu A)$. If $X$ is a locally compact metric space, then the
space of Radon measures on $X$ is considered as a topological space under the
weak-$*$ topology.

\section{Proof of Lemma~\ref{mainlemma}}\label{proofofmainlemma}

Let $Y$ be a topological space. A function $f\colon Y \to (-\infty, \infty]$ is
\emph{lower semicontinuous} if $f^{-1}\big((a, \infty]\big)$ is open for every
$a \in \rea$, or equivalently if $f(y_0) \leq \liminf_{y \to y_0} f(y)$ for
every $y_0 \in Y$. The following lemma is well known, but a proof is included
for the convenience of the reader.

\begin{lemma} \label{semicontlemma}
Let $X$ be a compact metric space and let $t > 0$. Then $\mu \mapsto I_t(\mu)$
is lower semicontinuous.

\begin{proof}
Let $\mathcal M(X)$ be the space of finite measures on $X$.
It will first be shown that the map
$\mathcal M(X) \to \mathcal M(X) \times \mathcal M(X)$,
$\mu \mapsto \mu \times \mu$ is continuous. Let $\eta$ be a continuous
function on $X \times X$ and let $\mu_0 \in \mathcal M(X)$. It suffices
to show that for every $\varepsilon > 0$, the set
	\[
	A =
	\left\{
	\mu \in \mathcal M(X); \,
	|(\mu \times \mu)(\eta) - (\mu_0 \times \mu_0)(\eta)| < 3 \varepsilon
	\right\}
	\]
contains an open neighbourhood of $\mu_0$.

Let $K > \mu_0(X)$.
By Stone--Weierstrass' theorem, there are functions $\{\eta_i\}_{i = 1}^n$
of the form $\eta_i = \eta_{i, 1} \times \eta_{i, 2}$, where $\eta_{i, j}$ are
continuous, such that
$\left\|\eta - \sum_{i = 1}^n \eta_i \right\|_\infty < \varepsilon / K^2$.
The set
	\[
	V = \left\{
	\mu \in \mathcal M(X); \,
	|(\mu\times\mu)(\eta_i) - (\mu_0\times \mu_0)(\eta_i)| <
	\frac \varepsilon n
	\text{ for every $i$ and } \mu(X) < K \right\}
	\]
is open since
$(\mu\times\mu)(\eta_i) = \mu(\eta_{i, 1})\mu(\eta_{i, 2})$,
and $\mu_0 \in V$. If $\mu \in V$ then
	\begin{align*}
	\left| (\mu \times \mu)(\eta) - (\mu_0 \times \mu_0)(\eta) \right|
	&\leq
	\left|
	(\mu\times\mu)(\eta) -
	(\mu \times \mu) \left(\sum_{i = 1}^n \eta_i \right)
	\right| \\
	&+
	\sum_{i = 1}^n
	\left|
	(\mu \times \mu)(\eta_i)
	-
	(\mu_0 \times \mu_0)(\eta_i)
	\right| \\
	&+
	\left|
	(\mu_0 \times \mu_0)\left( \sum_{i = 1}^n \eta_i \right)
	-
	(\mu_0\times\mu_0)(\eta)
	\right| < 3\varepsilon,
	\end{align*}
so that $V$ is a subset of $A$.

Define
	\[
	I_t^M(\mu) =
	\iint d(x, y)^{-t} \wedge M \intd{\mu(y)} \intd{\mu(x)}.
	\]
Let $a \in \rea$ and let $\mu_0$ be a measure in $\mathcal M(X)$
such that $I_t(\mu_0) > a$. Then there exists $M$ such that
$I_t^M(\mu_0) > a$, and the set $\left\{\mu; \, I_t^M(\mu) > a \right\}$
is open, contains $\mu_0$ and is contained in
$\left\{\mu; \, I_t(\mu) > a\right\}$.
\end{proof}
\end{lemma}

\begin{proof}[Proof of Lemma~\ref{mainlemma}]
Let $\varepsilon \in (0, \nu(X))$ and let $(\varepsilon_k)$ be a sequence of
positive numbers such that $\sum_k \varepsilon_k \leq \varepsilon$. Define
recursively a sequence $(n_k)_{k = 1}^\infty$ of natural numbers as follows,
using the notation $\rho_k = \varphi_{n_k} \cdot \ldots \cdot \varphi_{n_1}$
with the convention $\rho_0 = 1$. For $k \geq 1$, assume that
$n_1, \ldots, n_{k - 1}$ are defined. Then since
$\lim_{n \to \infty} \nu(\varphi_n \rho_{k - 1}) = \nu(\rho_{k - 1})$
and
	\[
	\liminf_{n \to \infty} I_t(\varphi_n\rho_{k - 1}) \leq
	I_t(\rho_{k - 1}),
	\]
it is possible to find $n_k > n_{k - 1}$ such that
	\[
	\nu(\rho_k) \geq \nu(\rho_{k - 1}) - \varepsilon_k
	\]
and
	\[
	I_t\left(\rho_k\right)
	\leq
	I_t\left(\rho_{k - 1}\right)
	+ \varepsilon_k.
	\]

Let $\mu$ be an accumulation point of the sequence of measures
$(\rho_k\intd\nu)$\,---\,then there is a strictly increasing
sequence $(k_i)$ such that $\mu = \lim_{i \to \infty} \rho_{k_i}\intd\nu$.
Thus
	\[
	\mu(X) = \lim_{i \to \infty} \nu(\rho_{k_i})
	\geq \liminf_{k \to \infty} \nu(\rho_k) \geq
	\nu(X) - \sum_k \varepsilon_k \geq \nu(X) - \varepsilon,
	\]
and by Lemma~\ref{semicontlemma},
	\[
	I_t(\mu) \leq
	\liminf_{i \to \infty} I_t(\rho_{k_i}) \leq
	\limsup_{k \to \infty} I_t(\rho_k) \leq
	I_t(\nu) + \sum_k \varepsilon_k \leq
	I_t(\nu) + \varepsilon.
	\]
Moreover,
	\[
	\supp \mu \subset
	\supp \nu \cap \bigcap_k \supp \varphi_{n_k}
	\subset
	\supp \nu \cap \limsup_n\left( \supp \varphi_n \right)
	\]
and thus
	\[
	\Cap_t\left(\supp \nu \cap \limsup_n (\supp \varphi_n)\right)
	\geq
	\frac{\mu(X)^2}{I_t(\mu)} \geq
	\frac{(\nu(X) - \varepsilon)^2}{I_t(\nu) + \varepsilon}.
	\]
Letting $\varepsilon \to 0$ concludes the proof.
\end{proof}

\section{Proof of Theorem~\ref{auxthm}}

The proof of Theorem~\ref{auxthm} is based on constructing a
sequence $(\varphi_n^\omega)$ of random continuous functions
supported on a compact neighbourhood of $W$, satisfying the hypothesis
of Lemma~\ref{mainlemma}, such that
	\[
	\limsup_n \left(\supp \varphi_n^\omega\right)
	\subset
	\limsup_n (\omega_n V_n).
	\]
A few lemmas are needed. The first one is used in the proofs of
Lemma~\ref{prop3.8lemma} and Theorem~\ref{auxthm}.

\begin{lemma} \label{ballenergyestlemma}
Let $(X, \lambda)$ be a $(c, d)$-regular space and let $t \in (0, d)$. Then
there is a constant $C$ such that for every $x \in X$ and $r > 0$,
	\[
	I_t(\oball x r) \leq C r^{2d - t}.
	\]

\begin{proof}
If $\varphi$ is a non-negative Borel function on $X$ then
	\[
	\int \varphi \intd\lambda =
	\int_0^\infty \lambda\left\{z \in X; \, \varphi(z) > \gamma \right\}
        \intd\gamma,
	\]
since both sides equal the $\lambda \times \mathcal L eb$-measure of the set
$\left\{ (z, u) \in X \times [0, \infty]; \, u \in (0, \varphi(z)) \right\}$
(note that the boundary of this set has measure $0$). Thus for $y \in X$,
	\begin{align*}
	\int_{\oball x r} d(&y, z)^{-t} \intd{\lambda(z)} =
	\int_0^\infty \lambda\left\{ z \in \oball x r; \, d(y, z)^{-t}
        > \gamma \right\} \intd\gamma \\
	&=
	\int_0^\infty \lambda\left(\oball x r \cap \oball{y}{\gamma^{-1/t}}
        \right) \intd\gamma
	\leq
	c \int_0^\infty \min\left(r^d, \, \gamma^{-d / t} \right) \intd\gamma \\
	&=
	c\int_0^{r^{-t}} r^d \intd\gamma + c\int_{r^{-t}}^\infty \gamma^{-d/t}
        \intd\gamma
	=
	\frac{cd}{d - t} r^{d - t},
	\end{align*}
and thus
	\[
	I_t\left( \oball x r \right)
	\leq
	\lambda\left(\oball x r\right) \cdot \frac{cd}{d - t} r^{d - t}
	\leq \frac{c^2d}{d - t} r^{2d - t}. \qedhere
	\]
\end{proof}
\end{lemma}

The next lemma is a variant of~\cite[Proposition 3.8]{FJJSUP}. 

\begin{lemma} \label{prop3.8lemma}
Let $(X, \lambda)$ be a locally compact\/ $(c, d)$-regular space and
let\/ $\theta$ be a finite compactly supported measure on $X$. Then there
is a sequence $(\varphi_n)$ of
non-negative continuous compactly supported functions on $X$ such that\/
$\lim_{n \to \infty} \varphi_n \intd\lambda = \theta$ (in the weak-$*$
topology) and\/ $\limsup_{n \to \infty} I_t(\varphi_n) \leq I_t(\theta)$ for every
$t \in (0, d)$.

\begin{proof}
For each $n$, let $\{x_{n, i}\}$ be a maximal (finite or countable)
$1 / n$-separated subset of $X$. Then $\left\{\oball{x_{n, i}}{1 / 2n}\right\}$
are disjoint and $\left\{\oball{x_{n, i}}{1 / n}\right\}$ is a cover of $X$.
Let $Q_{n, i}$ be the set of points $x$ in $X$ for which $i$ is the
first index such that $d(x, x_{n, i}) = \min_j d(x, x_{n, j})$, that is,
	\begin{align*}
	Q_{n, i} = \big\{
	x \in X; \, &d(x, x_{n, i}) < d(x, x_{n, j}) \text{ for }
	j = 1, \ldots, i - 1 \text{ and } \\
	&d(x, x_{n, i}) \leq d(x, x_{n, j}) \text{ for } j = i, \ldots, N_n
	\big\}.
	\end{align*}
Then, for each $n$, $\{Q_{n, i}\}$ is a partition of $X$ into Borel sets and
	\[
	\oball{x_{n, i}}{\frac{1}{2n}} \subset Q_{n, i}
	\subset \oball{x_{n, i}}{\frac 1 n}
	\]
for every $i$.

For each $i$, let
	\[
	\varphi_{n, i}(x) = a_{n, i} \max\left(
	0, 1 - 4n d(x, x_{n, i})
	\right),
	\]
where $a_{n, i}$ is such that $\lambda(\varphi_{n, i}) = 1$.
Then $\varphi_{n, i}$ is supported on $\cball{x_{n, i}}{1 / 4n}$, and
$a_{n, i} \leq 2^{1 + 3d} c n^d$ since
	\[
	\int_{\oball{x_{n, i}}{\frac{1}{4n}}}
	\left(1 - 4n d(x, x_{n, i}) \right)
	\intd{\lambda(x)}
	\geq
	\frac 1 2 \lambda\left(\oball{x_{n, i}}{\frac{1}{8n}}\right)
	\geq 2^{-(1 + 3d)} c^{-1} n^{-d}.
	\]
Let
	\[
	\varphi_n = \sum_i \theta(Q_{n, i}) \varphi_{n, i}.
	\]
Note that, for every $x$, there are only finitely many non-zero terms in
the sum defining $\varphi_n$. Further, $\varphi_n$ is compactly supported since
$\theta$ has a compact support.  
        
Let $\eta$ be a continuous compactly supported function on $X$ and let
$\varepsilon > 0$. Then $\eta$ is \emph{uniformly} continuous and there
is some $n_0$ such that
	\[
	\left|\eta(x) - \eta(x_{n, i})\right| \leq \frac{\varepsilon}{2\theta(X)}
	\]
whenever $n \geq n_0$ and $x \in Q_{n, i}$. Then for $n \geq n_0$,
	\begin{align*}
	&\left|\theta(\eta) - (\varphi_n\intd\lambda)(\eta)\right|
	\leq \\
	&\qquad \sum_i
	\left(
	\int_{Q_{n, i}} |\eta - \eta(x_{n, i})| \intd\theta
	+
	\theta(Q_{n, i})
	\int_{Q_{n, i}} |\eta(x_{n, i}) - \eta|\varphi_{n, i} \intd\lambda
	\right)
	\leq \varepsilon,
	\end{align*}
using that $(\varphi_{n, i}\intd\lambda)(Q_{n, i}) = 1$ for
every $i$. Thus $\lim_{n \to \infty} \varphi_n\intd\lambda = \theta$.
It remains to show that $\limsup_{n \to \infty} I_t(\varphi_n) \leq I_t(\theta)$,
and for this it may be assumed that $I_t(\theta) < \infty$.

Let $n$ be a natural number and $\alpha > 1$, and if $\psi_1, \psi_2$ are
continuous compactly supported functions on $X$ let
	\[
	J_t(\psi_1, \psi_2) = \iint d(x, y)^{-t}
	\psi_1(x) \psi_2(y) \intd{\lambda(x)} \intd{\lambda(y)}.
	\]
Then
	\[
	I_t(\varphi_n) = S_1 + S_2 + S_3,
	\]
where
	\begin{align*}
	S_1 &= \sum_{d(x_{n, i}, x_{n, j}) \geq \frac \alpha  n}
	\theta(Q_{n, i})\theta(Q_{n, j})
	J_t(\varphi_{n, i}, \varphi_{n, j}), \\
	S_2 &= \sum_{\substack{d(x_{n, i}, x_{n, j}) < \frac \alpha  n \\ i \neq j}}
	\theta(Q_{n, i})\theta(Q_{n, j})
	J_t(\varphi_{n, i}, \varphi_{n, j}), \\
	S_3 &= \sum_{\substack{i \\ \phantom{d(x_{n, i}, x_{n, j}) < \frac \alpha  n}}}
	\theta(Q_{n, i})^2 I_t(\varphi_{n, i}).
	\end{align*}

If $x \in \supp \varphi_{n, i}$ and $y \in \supp \varphi_{n, j}$ and
$d(x_{n, i}, x_{n, j}) \geq \alpha / n$, then
	\[
	|d(x, y) - d(x_{n, i}, x_{n, j})| \leq \frac 1 n \leq
	\frac{d(x_{n, i}, x_{n, j})}{\alpha},
	\]
so that
	\[
	\frac{\alpha - 1}{\alpha}
	\leq \frac{d(x, y)}{d(x_{n, i}, x_{n, j})} \leq
	\frac{\alpha + 1}{\alpha}.
	\]
Thus for $i, j$ appearing in $S_1$,
	\begin{align*}
	\theta(Q_{n, i})\theta(Q_{n, j}) J_t(\varphi_{n, i}, \varphi_{n, j})
	&\leq
	\left(\frac{\alpha}{\alpha - 1}\right)^t \theta(Q_{n, i})\theta(Q_{n, j})
	d(x_{n, i}, x_{n, j})^{-t} \\
	&\leq
	\left(\frac{\alpha + 1}{\alpha - 1}\right)^t
	\iint_{Q_{n, i} \times Q_{n, j}} d(x, y)^{-t} \intd{\theta(x)}
	\intd{\theta(y)},
	\end{align*}
and it follows that
	\[
	S_1 \leq \left(\frac{\alpha + 1}{\alpha - 1}\right)^t I_t(\theta).
	\]

If $x \in \supp \varphi_{n, i}$ and $y \in \supp \varphi_{n, j}$ 
then $i \neq j$ implies that $d(x, y) \geq 1 / 2n$ and if
$x \in Q_{n, i}$ and $y \in Q_{n, j}$ then $d(x_{n, i}, x_{n, j}) \leq \alpha / n$
implies that $d(x, y) \leq (\alpha + 2) / n$. Thus
	\begin{align*}
	S_2 &\leq
	2^t \sum_{\substack{d(x_{n, i}, x_{n, j}) < \frac \alpha  n \\ i \neq j}}
	\theta(Q_{n, i})\theta(Q_{n, j}) n^t \\
	&\leq
	2^t (\alpha + 2)^t
	\sum_{d(x_{n, i}, x_{n, j}) < \frac \alpha n}
	\iint_{Q_{n, i} \times Q_{n, j}} d(x, y)^{-t} \intd{\theta(x)}\intd{\theta(y)}\\
	&\leq
	2^t (\alpha + 2)^t \iint_{d(x, y) \leq \frac{\alpha + 2}{n}}
	d(x, y)^{-t} \intd{\theta(x)} \intd{\theta(y)}.
	\end{align*}

By Lemma~\ref{ballenergyestlemma} there is a constant C such that
	\[
	I_t(\oball x r) \leq C r^{2d - t} 
	\]
for every $x \in X$ and $r > 0$. It follows that
	\[
	I_t(\varphi_{n, i}) \leq a_{n, i}^2 I_t\left( \oball{x_{n, i}}{\frac{1}{4n}}
        \right)
	\leq
	C' n^t,
	\]
where $C' = 4^{1 + 3d} c^2 C$, so that
	\begin{align*}
	S_3 &\leq C' \sum_i \theta(Q_{n, i})^2 n^t
	\leq 2^t C' \sum_i \iint_{Q_{n, i} \times Q_{n, i}}
	d(x, y)^{-t} \intd{\theta(x)}\intd{\theta(y)} \\
	&\leq
	2^t C'
	\iint_{d(x, y) \leq \frac 2 n}
	d(x, y)^{-t} \intd{\theta(x)}\intd{\theta(y)}.
	\end{align*}

Given $\varepsilon > 0$ it is possible to choose $\alpha$ large
enough so that \mbox{$S_1 \leq I_t(\theta) + \varepsilon$}, and then $n_0$ large
enough so that $S_2 + S_3 \leq \varepsilon$ for every $n \geq n_0$.
It follows that
$\limsup_{n \to \infty} I_t(\varphi_n) \leq I_t(\theta) + 2\varepsilon$,
and letting $\varepsilon \to 0$ concludes the proof.
\end{proof}
\end{lemma}

The following lemma is a modification of the argument from
\cite[p.~1580]{FJJSUP}.

\begin{lemma} \label{diamVnto0lemma}
Let $(X, \lambda)$ be a $(c, d)$-regular space and let
$(V_n)$ be a sequence of open subsets of an open ball $B$ in $X$
such that\/ $\sum_n \Cap_t(V_n) = \infty$. Then there is a
sequence $(V_n')$ of open subsets of\/ $X$ such that

\begin{enumerate}[label=\roman*)]
\item
$V_n' \subset V_n$ for every $n$,

\item
$\lim_{n \to \infty} |V_n'| = 0$, and

\item
$\sum_n \Cap_t(V_n') = \infty$.
\end{enumerate}

\begin{proof}
Let $B_1 = \{x \in X; \, d(x, B) < 1 \}$ and for each $n$, let $\mu_n$ be a
probability measure on $V_n$ such that $I_t(\mu_n) \leq 2\Cap_t(V_n)^{-1}$.
Consider any $r \in (0, 1)$ and let
$A_n = \{x \in X; \, \mu_n(\oball x r) > 0\}$,
and for $x \in A_n$ let
$\mu_n^x = \mu_{n}(\oball x r)^{-1} \restr{\mu_n}{\oball x r}$.
Using Cauchy--Schwartz' inequality, Fubini's theorem and then that
$\oball y r \subset B_1$ for $\mu_n$-almost every
$y$ for every $n$,
	\begin{align*}
	\int_{B_1} \mu_n(\oball x r)^2 \intd{\lambda(x)}
	&\geq
	\frac{1}{\lambda(B_1)} \left(
	\int_{B_1} \mu_n(\oball x r) \intd{\lambda(x)}\right)^2 \\
	&=
	\frac{1}{\lambda(B_1)} \left(
	\int \lambda\left(B_1 \cap \oball y r\right)
	\intd{\mu_n(y)}\right)^2
	\geq
	\frac{c^{-2}r^{2d}}{\lambda(B_1)}.
	\end{align*}
Thus for any natural number $a$,
	\begin{align*}
	\int_{B_1} \sum_{n = a}^\infty
	\Cap_t\left(V_n \cap \oball x r\right) \intd{\lambda(x)}
	&\geq
	\int_{B_1} \sum_{\substack{n \geq a \\ x \in A_n}}
	\frac{1}{I_t\left(\mu_n^x\right)} \intd{\lambda(x)} \\
	&\geq
	\sum_{n = a}^\infty \frac{1}{I_t(\mu_n)} \int_{B_1 \cap A_n}
	\mu_n(\oball x r)^2 \intd{\lambda(x)} \\
	&=
	\sum_{n = a}^\infty \frac{1}{I_t(\mu_n)} \int_{B_1}
	\mu_n(\oball x r)^2 \intd{\lambda(x)} = \infty.
	\end{align*}
It follows that $\sum_{n = a}^\infty \Cap_t\left(V_n \cap \oball x r\right)$
is unbounded as a function of $x$, and hence there exist $x \in B_1$ and a
natural number $b$ such that
	\[
	\sum_{n = a}^b
	\Cap_t\left( V_n \cap \oball x r\right) \geq 1.
	\]

It is now possible to recursively define a strictly increasing sequence $(n_k)$
of natural numbers and a sequence $(x_k)$ of points in $B_1$, such that
$n_1 = 1$ and for every $k$,
	\[
	\sum_{n = n_k}^{n_{k + 1} - 1}
	\Cap_t\left( V_n \cap \oball{x_k}{2^{-k}} \right) \geq 1.
	\]
The sequence $(V_n')$ defined by $V_n' = V_n \cap \oball{x_k}{2^{-k}}$
for $n = n_k, \ldots, n_{k + 1} - 1$ has the properties \emph{i)--iii)}.
\end{proof}
\end{lemma}

\begin{lemma} \label{clemma}
Let $(b_n)$ be a sequence of positive numbers bounded away from $0$,
such that $\sum_n b_n^{-1} = \infty$. Then there are non-negative
numbers $(a_{n, k})_{n, k \in \nat}$ such that

\begin{enumerate}[label=\roman*)]
\item \label{itemi}
$(a_{n, k})_k$ has finite support for every $n$ and\/
$\lim_{n \to \infty} \min \{k; \, a_{n, k} \neq 0\} = \infty$,

\item \label{itemii}
$\sum_k a_{n, k} = 1$ for every $n$ and\/
$\sum_{n, k} a_{n, k}^2 < \infty$, and

\item \label{itemiii}
$\lim_{n \to \infty} \sum_k a_{n, k}^2 b_k  = 0$.

\end{enumerate}

\begin{proof}
Let
	\[
	a_{n, k} = 
	\begin{cases}
	a_n b_k^{-1} & \text{if } M_n \leq k \leq N_n, \\
	0 & \text{otherwise},
	\end{cases}
	\]
where
	\[
	a_n = \frac{1}{\sum_{k = M_n}^{N_n} b_k^{-1}}.
	\]
Since $\sum_n b_n^{-1} = \infty$ it is possible to choose $(M_n)$
and $(N_n)$ such that~{\it\ref{itemi}} holds and
$\sum_n a_n < \infty$. Then clearly $\sum_k a_{n, k} = 1$ for
every $n$, and
	\[
	\sum_{n, k} a_{n, k}^2 \leq
	B \sum_n \left( a_n \sum_k a_{n, k} \right) =
	B \sum_n a_n < \infty,
	\]
where $B = \sup_k b_k^{-1}$. Finally,
	\[
	\sum_k a_{n, k}^2 b_k
	=
	\frac{\sum_{k = M_n}^{N_n} b_k^{-1}}
	{\left(\sum_{k = M_n}^{N_n} b_k^{-1} \right)^2}
	=
	a_n,
	\]
which converges to $0$ when $n \to \infty$.
\end{proof}
\end{lemma}

\begin{lemma} \label{asconvlemma}
Let $\mu$ be a probability measure on a compact metric
space $X$, and define the probability space $(\Omega, \P)$ by\/
$\Omega = X^\nat$ and\/ $\P = \mu^\nat$. Let\/ $(a_{n, k})$ be non-negative
numbers such that\/ $\sum_k a_{n, k} = 1$ for every $n$ and\/
$\sum_{n, k} a_{n, k}^2 < \infty$. For $\omega \in \Omega$, let
	\[
	\mu_n^\omega = \sum_k a_{n, k} \delta_{\omega_k}.
	\]
Then almost surely\/ $\lim_{n \to \infty} \mu_n^\omega = \mu$.

\begin{proof}
Let $\eta$ be a continuous function on $X$. Then for every $k$,
	\[
	\E \eta(\omega_k) = \mu(\eta)
	\]
and
	\[
	\Var \eta(\omega_k) = \mu\left(\eta^2\right) - \mu(\eta)^2
        \leq \|\eta\|_\infty^2,
	\]
and hence
	\[
	\E(\mu_n^\omega(\eta)) = \mu(\eta)
	\]
and
	\[
	\Var(\mu_n^\omega(\eta)) = \sum_k a_{n, k}^2 \Var \eta(\omega_k)
	\leq \|\eta\|_\infty^2 \sum_k a_{n, k}^2.
	\]

Let $\varepsilon > 0$. Then by Chebyshev's inequality,
	\[
	\sum_n \P\{|\mu_n^\omega(\eta) - \mu(\eta)| \geq \varepsilon\}
	\leq
	\sum_n \frac{\|\eta\|_\infty^2 \sum_k a_{n, k}^2}{\varepsilon^2} < \infty,
	\]
so that Borel--Cantelli's lemma implies that
	\[
	\limsup_{n \to \infty} |\mu_n^\omega(\eta) - \mu(\eta)| \leq \varepsilon
	\]
almost surely. Since the space of continuous functions on $X$ is separable, it
follows that $\lim_{n \to \infty} \mu_n^\omega = \mu$ almost surely.
\end{proof}
\end{lemma}

\begin{lemma} \label{asliminflemma}
Let $(\xi_n)$ be a sequence of independent random variables.
Then almost surely
	\[
	\liminf_{n \to \infty} \xi_n \leq
	\liminf_{n \to \infty} \E \xi_n.
	\]

\begin{proof}
By taking a subsequence it may be assumed that $\lim_{n \to \infty} \E\xi_n$
exists. Let $(\varepsilon_n)$ be a sequence of positive numbers converging
to $0$, such that
	\[
	\sum_n \frac{\varepsilon_n}{1 + \varepsilon_n} = \infty.
	\]
By Markov's inequality, 
	\[
	\P\left(\xi_n \leq (1 + \varepsilon_n)\E\xi_n \right)
	=
	1 - \P\left(\xi_n > (1 + \varepsilon_n)\E\xi_n \right)
	\geq
	1 - \frac{\E \xi_n}{(1 + \varepsilon_n)\E \xi_n}
	=
	\frac{\varepsilon_n}{1 + \varepsilon_n}.
	\]
Then by Borel--Cantelli's lemma there is almost surely a strictly
increasing sequence $(n_k)$ of natural numbers such that
$\xi_{n_k} \leq (1 + \varepsilon_{n_k})\E\xi_{n_k}$ for every $k$,
and thus
	\[
	\liminf_{n \to \infty} \xi_n
	\leq
	\liminf_{k \to \infty} \xi_{n_k}
	\leq
	\liminf_{k \to \infty} \E\xi_{n_k}
	\leq
	\limsup_{n \to \infty} \E\xi_n
	=
	\liminf_{n \to \infty} \E\xi_n.
	\qedhere
	\]
\end{proof}
\end{lemma}

\begin{proof}[Proof of Theorem~\ref{auxthm}]
By Lemma~\ref{diamVnto0lemma} it may be assumed that
$\lim_{n \to \infty} |V_n| = 0$.

Define a sequence $(\varphi_n^\omega)$ of random continuous functions on $G$ in
the following way. For each $n$, there is a probability measure $\theta_n$
on $V_n$ such that $I_t(\theta_n) \leq 2 \Cap_t(V_n)^{-1}$, and an open
subset $A_n$ of $V_n$ such that $d(A_n, V_n^c) > 0$ and
$\theta_n(A_n) \geq 1 / 2$. By
Lemma~\ref{prop3.8lemma} there is then a non-negative continuous
function $\psi_n'$ on $G$ such that
$(\psi_n'\intd\lambda)(A_n) \geq \theta_n(A_n) / 2 \geq 1 / 4$
and $I_t(\psi_n') \leq 2 I_t(\theta_n) \leq 4 \Cap_t(V_n)^{-1}$. Let $\psi_n''$ be
a continuous function on $G$ such that
$\chi_{A_n} \leq \psi_n'' \leq \chi_{V_n}$ and let
$\psi_n = c_n \psi_n' \psi_n''$, where $c_n$ is such that
$\psi_n\intd\lambda$ is a probability measure. Then $c_n \leq 4$ and
hence
	\[
	I_t(\psi_n) \leq c_n^2 I_t(\psi_n') \leq
	64 \Cap_t(V_n)^{-1},
	\]
so that the hypothesis of the theorem implies that
$\sum_n I_t(\psi_n)^{-1} = \infty$. Let
$(a_{n, k})$ be as in Lemma~\ref{clemma} with respect to
$b_n = I_t(\psi_n) \,\, (\geq |V_n|^{-t})$, and let
	\[
	\varphi_n^\omega = \sum_k a_{n, k} \psi_k^\omega,
	\]
where $\psi_k^\omega(x) = \psi_k\left(\omega_k^{-1} x\right)$.

Property~{\it\ref{itemii}} of Lemma~\ref{clemma} together with
Lemma~\ref{asconvlemma} applied in the space $(\overbar W, \restr\lambda W)$
implies that for almost every $\omega$,
	\begin{equation}\label{niceconvergence}
	\lim_{n \to \infty} \sum_k a_{n, k} \delta_{\omega_k} = \restr\lambda W.
	\end{equation}
Consider an arbitrary non-negative continuous function $\eta$ on $G$ such
that $\supp \eta \subset W$ and let $\nu = \eta\intd\lambda$ (specific choices
for $\eta$ will be made later).  
Since $\lim_{n \to \infty} |V_n| = 0$, it follows for all $\omega$ satisfying
\eqref{niceconvergence} that
$\lim_{n \to \infty} \varphi_n^\omega\intd\lambda = \restr \lambda W$
and thus
	\[
	\lim_{n \to \infty} \varphi_n^\omega \intd\nu =
	\eta \intd{\restr\lambda W} = \nu.
	\]

Let $\rho$ be a continuous function on $G$ with compact support.
Using that $\psi_i^\omega(x)$ and $\psi_j^\omega(y)$ are
independent for $i \neq j$,
	\begin{align*}
	\E I_t(\varphi_n^\omega&\rho\intd\nu) =
	\sum_k a_{n, k}^2 \E I_t(\psi_k^\omega\rho\intd\nu) \\
	&+ \sum_{i \neq j} a_{n, i} a_{n, j}
	\iint d(x, y)^{-t} \E \psi_i^\omega(x)
	\E \psi_j^\omega(y) \rho(x) \rho(y)
	\intd{\nu(x)}\!\intd{\nu(y)}.
	\end{align*}
Since $I_t(\psi_k^\omega\rho\intd\nu) \leq
\|\rho\|_\infty^2 \|\eta\|_\infty^2 I_t(\psi_k)$,
it follows from property~{\it\ref{itemiii}} of Lemma~\ref{clemma}
that the first sum converges to $0$ when $n \to \infty$.
Next,
	\[
	\E \psi_i^\omega(x) =
	\int \psi_i\left(\omega_i^{-1} x\right) \intd{\restr\lambda W(\omega_i)}
	\leq
	\int \psi_i\left(\omega_i^{-1} x\right) \intd{\lambda(\omega_i)} =
	\lambda(\psi_i) = 1,
	\]
using that $\lambda$ is invariant under right translation
and inversion. Thus the integral in the second sum is less than
or equal to $I_t(\rho\intd\nu)$ and it follows by property~{\it\ref{itemii}}
in Lemma~\ref{clemma} that the second sum is less than or equal to
$I_t(\rho\intd\nu)$ as well. Thus by Lemma~\ref{asliminflemma}, almost
surely
	\[
	\liminf_{n \to \infty}
	I_t(\varphi_n^\omega\rho\intd\nu)
	\leq
	I_t(\rho\intd\nu).
	\]
Lemma~\ref{mainlemma} applied in the space $(\overbar W, \nu)$ together
with \eqref{caplesscontent} now implies that almost surely
	\[
	\hmeas^t\left( \supp \eta \cap E(\omega) \right)
	\geq
	\frac{\lambda(\eta)^2}{I_t(\eta)}.
	\]

For $m = 1, 2, \ldots$,  let $\mathcal B_m$ be a maximal collection
of disjoint open balls in $W$ of radius $2^{-m}$. For each
$B \in \bigcup_m \mathcal B_m$, let $\eta_B$ be a non-negative continuous
function on $G$ such that $\chi_{\frac 1 2 B} \leq \eta_B \leq \chi_B$,
where $\frac 1 2 B$ is the ball concentric with $B$ having half the radius.
Since $\bigcup_m \mathcal B_m$ is countable, almost every $\omega$ is
such that whenever $B \in \mathcal B_m$ then
	\[
	\hmeas^t(B \cap E(\omega)) \geq
	\frac{\lambda(\eta_B)^2}{I_t(\eta_B)} \geq
	C 2^{-tm},
	\]
where the last inequality holds by Lemma~\ref{ballenergyestlemma} for some
constant $C$ that is independent of $m$.

Let $\omega$ be such and let $\mathcal U$ be an open subset of $W$.
Since $(G, \lambda)$ is $d$-regular, there is a positive constant $C'$
and some $m_0$ such that if $m \geq m_0$ then
	\[
	\#\{B \in \mathcal B_m; \, B \subset \mathcal U\} \geq C' 2^{dm}.
	\]
Thus for $m \geq m_0$,
	\[
	\hmeas^t(\mathcal U \cap E(\omega)) \geq
	\sum_{\substack{B \in \mathcal B_m \\ B \subset \mathcal U\phantom{_m}}}
	\hmeas^t(B \cap E(\omega)) \geq
	C C' 2^{(d - t)m},
	\]
and letting $m \to \infty$ shows that
$\hmeas^t(\mathcal U \cap E(\omega)) = \infty$.
\end{proof}
  
\section{Hausdorff content and energy of rectangles in \heis}
\label{contentcapacitysection}

The purpose of this section is to estimate the Hausdorff content and energy
of a rectangle $\crect{x}{r}$ in the Heisenberg group, up to multiplicative
constants. Only upper bounds are provided, but it follows
from~\eqref{caplesscontent}
that they are the best possible ones. The multiplicative constants will mostly
be implicit, using the following notation. If $e_1$ and $e_2$ are expressions
depending on some parameters, then $e_1 \lesssim e_2$ means that there is a
positive constant $C$ such that $e_1 \leq Ce_2$ for all parameter values.
Often some of
the parameters in $e_1$ and $e_2$ will be considered as constants\,---\,then $C$
may depend on those parameters. For example, the implicit constants always
depend on $t$.

\subsection{Upper bound for the Hausdorff content of a rectangle}

\begin{lemma} \label{contentestlemma}
For all $t\in [0,4]$,
         \[
         \hmeas_\infty^t(\crect 0 r) \lesssim \Phi^t(r),
         \] 
where the implicit constant depends on $t$ but not on $r$.
\end{lemma}

\begin{proof}\let\qed\relax % don't want a qed symbol at the end of this proof
Let $R = \crect 0 r$. It is enough to show that if $r_1 \leq r_2$, then
	\[
	\hmeas_\infty^t(R) \lesssim
	\begin{cases}
	r_2^t & \text{if } t \in [0, 2],\\
	r_1^{t - 2} r_2^2& \text{if } t \in [2, 4],
	\end{cases}
	\]
and if $r_1 \geq r_2$, then 
	\[
	\hmeas_\infty^t(R) \lesssim
	\begin{cases}
	r_1^t & \text{if } t \in [0, 3], \\
	r_1^{6 - t} r_2^{2(t - 3)} & \text{if } t \in [3, 4].
	\end{cases}
	\]

For $t \geq 0$,
	\[
	\hmeas_\infty^t(R) \leq |R|^t \lesssim
	\max\left(r_1^t, r_2^t\right) =
	\begin{cases}
	r_2^t & \text{if } r_1 \leq r_2, \\
	r_1^t & \text{if } r_1 \geq r_2.
	\end{cases}
	\]
The vertical segment $S = \{0\} \times \left[-r_2^2, r_2^2\right]$
is $r_1$-dense in $R$, and
	\[
	D = \left\{(0, kr_1^2); \,
	k = -\floor{\frac{r_2^2}{r_1^2}}, \ldots, \floor{\frac{r_2^2}{r_1^2}}
	\right\}
	\]
is $r_1$-dense in $S$\,---\,thus $D$ is $2r_1$-dense in $R$. If $r_1 \leq r_2$
then $\# D \lesssim r_2^2 / r_1^2$, and hence
	\[
	\hmeas_\infty^t(R) \lesssim
	\frac{r_2^2}{r_1^2} \cdot r_1^t = r_1^{t - 2} r_2^2.
	\]

It remains to show that $\hmeas_\infty^t(R) \lesssim r_1^{6 - t} r_2^{2(t - 3)}$
for $r_1 \geq r_2$ and $t \in [3, 4]$. This is done by estimating the
Hausdorff content of annuli of the form
	\[
	A_\rho = \left\{
	(x, y, z); \, \left|\left(x^2 + y^2\right)^{1/2} - \rho\,\right|
	\leq \frac{r_2^2}{2\rho},
	\,\,
	|z| \leq r_2^2
	\right\}.
	\]
Let
	\[
	C_\rho = \{ (x, y, 0); \, x^2 + y^2 = \rho^2 \}.
	\]
\end{proof}

\begin{sublemma}
For $\rho \geq r_2$, the set $C_\rho$ is $\sqrt 2 r_2^2 / \rho$-dense
in $A_\rho$.

\begin{proof}
Let $p = (\rho, 0, 0)$ and let 
	\[
	S = \{(x, y, z); \, x \geq 0, \, z = 2\rho y\};
	\]
this is the part of the plane $H(p)$ defined in the introduction
where the $x$-coordinate is non-negative. Let
$q = (x, y, z)$ be a point in $A_\rho \cap S$. Then
	\[
	x \leq \rho + \frac{r_2^2}{2\rho}
	\qquad\qquad\text{and}\qquad\qquad
	|y| = \frac{|z|}{2\rho} \leq \frac{r_2^2}{2\rho},
	\]
and
	\[
	x^2 + y^2 \geq \left(\rho - \frac{r_2^2}{2\rho}\right)^2,
	\]
so that
	\begin{align*}
	x \geq \left(
	\left( \rho - \frac{r_2^2}{2\rho} \right)^2 - y^2
	\right)^{1 / 2}
	&\geq
	\left(
	\left( \rho - \frac{r_2^2}{2\rho} \right)^2 - \left(\frac{r_2^2}{2\rho}
        \right)^2
	\right)^{1 / 2} \\
	&=
	\left(\rho^2 - r_2^2 \right)^{1 / 2}
	\geq
	\rho - \frac{r_2^2}{\rho}
	\end{align*}
(the last inequality is proved by squaring both sides and using
that $\rho \geq r_2$). Thus
	\[
	\dheis(p ,q) = \left((x - \rho)^2 + y^2 \right)^{1 / 2}
	\leq \sqrt 2 \max\left(|x - \rho|, |y| \right)
	\leq \frac{\sqrt 2 r_2^2}{\rho}.
	\]
Let $R_\alpha$ be the rotation by $\alpha$ around the vertical axis
$\{0\} \times \rea$. The statement follows since $\dheis$ is invariant
under $R_\alpha$ and
	\[
	C_\rho = \bigcup_\alpha R_\alpha(p)
	\qquad\quad\text{and}\quad\qquad
	A_\rho = \bigcup_\alpha R_\alpha\left( A_\rho \cap S \right).
	\qedhere
	\]
\end{proof}
\end{sublemma}

\begin{sublemma}
Let $\varepsilon > 0$. The set
	\[
	D_\rho = \left\{
	\left(\rho \cos\left( \frac{k \varepsilon^2}{2\rho^2}\right), \,
	\rho \sin\left( \frac{k \varepsilon^2}{2\rho^2}\right), \, 0 \right);
	\,
	k = 0, \ldots, \floor{\frac{4\pi\rho^2}{\varepsilon^2}}
	\right\}
	\]
is $\varepsilon$-dense in $C_\rho$.

\begin{proof}
The points $p = (\rho, 0, 0)$ and $q = (\rho\cos \alpha, \rho \sin \alpha, 0)$
satisfy
	\begin{align*}
	\dheis(p, q) &= \rho \left(
	\left(\left(1 - \cos\alpha\right)^2 + \sin^2 \alpha \right)^2
	+ \left(2\sin \alpha\right)^2
	\right)^{1 / 4} \\
	&=
	2^{3 / 4} \rho
	\left(1 - \cos\alpha \right)^{1 / 4}
	\leq
	\rho |2\alpha|^{1 / 2},
	\end{align*}
using that $\cos\alpha \geq 1 - \alpha^2 / 2$. The same bound holds for
any pair of points $p, q$ on $C_\rho$ making an angle $\alpha$, and taking
$\alpha = \varepsilon^2 / 2\rho^2$ gives $\dheis(p, q) \leq \varepsilon$.
\end{proof}
\end{sublemma}

\begin{proof}[Proof of Lemma~\ref{contentestlemma} (continued)]
Take $\varepsilon = r_2^2 / \rho$ in the definition of $D_\rho$.
Then for $\rho \geq r_2$ the set $D_\rho$ is $3r_2^2/\rho$-dense
in $A_\rho$ and $\# D_\rho \lesssim \rho^4 / r_2^4$, and thus
	\[
	\hmeas_\infty^t(A_\rho) \lesssim \frac{\rho^4}{r_2^4}
	\cdot \left( \frac{r_2^2}{\rho} \right)^t
	= r_2^{2t - 4} \rho^{4 - t}.
	\]
Let $\rho_k = r_2 \sqrt k$. Then
	\[
	\rho_{k + 1} - \rho_k =
	r_2 \int_k^{k + 1} \frac{1}{2 \sqrt u} \intd u
	\leq \frac{r_2}{2 \sqrt k} = \frac{r_2^2}{2 \rho_k},
	\]
so that $A_{\rho_k}$ and $A_{\rho_{k + 1}}$ overlap. Thus
	\[
	R \subset \crect{0}{(r_2, r_2)} \cup
	\bigcup_{k = 1}^{\ceil{\left(\frac{r_1}{r_2}\right)^2}} A_{\rho_k}.
	\]
It follows that
	\begin{align*}
	\hmeas_\infty^t(R) &\lesssim
	r_2^t +
	\sum_{k = 1}^{\ceil{\left(\frac{r_1}{r_2}\right)^2}}
	r_2^{2t - 4} \rho_k^{4 - t}
	=
	r_2^t \left( 1 +
	\sum_{k = 1}^{\ceil{\left(\frac{r_1}{r_2}\right)^2}}
	k^{(4 - t)/2}  \right) \\
	&\lesssim
	r_2^t \left( 1 +
	\left( \frac{r_1}{r_2} \right)^{6 - t} \right)
	\lesssim
	r_1^{6 - t}r_2^{2(t - 3)},
	\end{align*}
using in the last step that $(r_1 / r_2)^{6 - t} \geq 1$. This completes the
proof of Lemma~\ref{contentestlemma}.
\end{proof}

\subsection{Upper bound for the energy of a rectangle}

\begin{lemma} \label{energyestlemma}
For all $t\in(0,4)\setminus\{1,2,3\}$,
        \[
        \Phi^t(r) \lesssim \Cap_t\left(\crect x r \right),
        \]
where the implicit constant depends on $t$ but not on $r$.
\end{lemma}

\begin{proof}\let\qed\relax % don't want a qed symbol at the end of this proof
Let $R = \crect 0 r$. Since
        \[
	\Cap_t(R) \geq \frac{\lambda(R)^2}{I_t(R)},
	\]
it is enough to show that if $r_1 \leq r_2$, then
	\[
	I_t(R) \lesssim
	\begin{cases}
	r_1^4r_2^{4 - t} & \text{if } t \in (0, 2), \\
	r_1^{6 - t}r_2^2 & \text{if } t \in (2, 4),
	\end{cases}
	\]
and if $r_1 \geq r_2$, then
	\[
	I_t(R) \lesssim
	\begin{cases}
	r_1^{4 - t}r_2^4 & \text{if } t \in (0, 3) \setminus \{ 1 \}, \\
	r_1^{t - 2}r_2^{2(5 - t)} & \text{if } t \in (3, 4).
	\end{cases}
	\]

Let
	\[
	R_t(p) = \int_R \dheis(p, q)^{-t} \intd{\lambda(q)},
	\]
so that
	\[
	I_t(R) = \int_R R_t(p) \intd{\lambda(p)}.
	\]
Since $R$, \dheis and $\lambda$ are invariant under rotation around the vertical
axis, the integral defining $R_t(p)$ does not depend on the angle of $p$ in the
horizontal plane. To estimate $R_t(p)$ it is therefore enough to
consider $p$ of the form $p = (\rho, 0, z_0)$. Assume that $\rho \in [0, r_1]$
and $z_0 \in [-r_2^2, r_2^2]$, so that $p \in R$.

Let
	\[
	f_\rho(x, y, z) = \max\left(
	|x|, |y|, |z - 2 \rho y|^{1 / 2}
	\right).
	\]
Then for $q = (x, y, z)$,
	\[
	\dheis(p, q) = \left( \left((x - \rho)^2 + y^2\right)^2 +
	(z - z_0 - 2\rho y)^2 \right)^{1 / 4}
	\approx
	f_\rho(x - \rho, y, z - z_0).
	\]
Define the Euclidean rectangles
	\begin{align*}
	&A = \big[-2r_1, 2r_1\big] \times \big[-2r_1, 2r_1\big] \times
	\big[-2 r_2^2, 2 r_2^2\big], \\
	&A' = (\rho, 0, z_0) + A,
	\end{align*}
where $+$ means Euclidean translation, and let
$B(a) = \{q; \, f_\rho(q) \leq a\}$. Since $R \subset A'$,
	\begin{align}
	\nonumber
	R_t(p) &\leq
	\int_{A'} \dheis(p, q)^{-t} \intd{\lambda(q)}
	\lesssim
	\int_{A'} f_\rho(x - \rho, y, z - z_0)^{-t} \intd{\lambda(x, y, z)} \\
	\label{Ptesteq}
	&=
	\int_{A} f_\rho(x, y, z)^{-t} \intd{\lambda(x, y, z)}
	=
	\int_0^\infty \lambda\left\{
	q \in A; \, f_\rho(q)^{-t} \geq \gamma \right\} \intd \gamma \\
	\nonumber
	&=
	\int_0^\infty
	\lambda\left(A \cap B\left(\gamma^{-1/t}\right)\right) \intd \gamma
	\approx
	\int_0^\infty \lambda\left(A \cap B(a)\right) a^{-(t + 1)} \intd a.
	\end{align}
To estimate $R_t(p)$ it is useful to have an upper bound
for $\lambda\left(A \cap B(a)\right)$.

The set $B(a)$ is the intersection
of the vertical cylinder $[-a, a] \times [-a, a] \times \rea$
with the set of points having vertical Euclidean distance at most $a^2$
to the plane $z = 2\rho y$. In particular, the projection of $B(a)$ to
the $yz$-plane is the intersection of the strips
	\[
	S_1 = \{(y, z); \, -a \leq y \leq a\},
	\qquad\qquad
	S_2 = \{(y, z); \, 2\rho y - a^2 \leq z \leq 2\rho y + a^2 \}.
	\]
The projection of $A$ to the $yz$-plane is the intersection of the strips
	\[
	S_3 = \{(y, z); \, -2r_1 \leq y \leq 2r_1\},
	\qquad\qquad
	S_4 = \{(y, z); \, -2r_2^2 \leq z \leq 2r_2^2\}.
	\]

It is used in the computations below that if $u \neq -1$
and $0 \leq v \leq w \leq \infty$ then
	\begin{equation} \label{intesteq}
	\int_v^w a^u \intd a \lesssim
	\max\left(v^{u + 1}, w^{u + 1} \right),
	\end{equation}
with the convention that $1 / 0 = \infty$, and where the implicit constant
depends on $u$ but not on $v$, $w$.
\end{proof}

\begin{sublemma}
It holds that $\lambda(A \cap B(a)) \lesssim \min
\left(a^4, \, r_1^2a^2, \, r_1^2r_2^2 \right)$.

\begin{proof}
The projection of $A \cap B(a)$ to the $yz$-plane is contained in
each of the parallelograms
	\[
	S_1 \cap S_2, \qquad\qquad
	S_3 \cap S_4, \qquad\qquad
	S_2 \cap S_3.
	\]
The first two have area $4a^3$ and $16 r_1r_2^2$, respectively,
and the third has vertices
	\[
	\left(-2r_1, -4\rho r_1 \pm a^2 \right),
	\qquad\qquad
	\left(2r_1, 4\rho r_1 \pm a^2 \right),
	\]
and hence area $8r_1a^2$. The extension of $A$ in the $x$-direction
is $4r_1$ and the extension of $B(a)$ in the $x$-direction is $2a$.
Thus
	\[
	\lambda(A \cap B(a)) \lesssim
	\min\left(a^3, \, r_1r_2^2, r_1a^2\right) \cdot
	\min\left(r_1, a\right)
	\leq
	\min\left(a^4, \, r_1^2a^2, \, r_1^2r_2^2 \right). \qedhere
	\]
\end{proof}
\end{sublemma}

\begin{proof}[Proof of Lemma~\ref{energyestlemma} (continued). The case 
$r_1 \leq r_2$] \let\qed\relax
Let $t \in (0, 4) \setminus \{ 2 \}$. Using \eqref{Ptesteq},
the sublemma and~\eqref{intesteq},
	\begin{align*}
	R_t(p) &\lesssim
	\int_0^\infty
	\min\left(a^4, \, r_1^2a^2, \, r_1^2r_2^2 \right) a^{-(t + 1)} \intd a \\
	&\leq
	\int_0^{r_1} a^{3 - t} \intd a
	+
	\int_{r_1}^{r_2} r_1^2a^{1 - t} \intd a
	+
	\int_{r_2}^\infty r_1^2r_2^2 a^{-(t + 1)} \intd a \\
	&\lesssim
	\max\left(
	r_1^{4 - t}, \, r_1^2 r_2^{2 - t}
	\right).
	\end{align*}
It follows that
	\[
	I_t(R) \lesssim \lambda(R) \cdot
	\max\left(r_1^{4 - t}, \, r_1^2 r_2^{2 - t}	\right)
	=
	\max\left(r_1^{6 - t}r_2^2, \, r_1^4 r_2^{4 - t} \right)
	=
	\begin{cases}
	r_1^4r_2^{4 - t} & \text{if } t \in (0, 2), \\
	r_1^{6 - t}r_2^2 & \text{if } t \in (2, 4).
	\end{cases}
	\]
\end{proof}

\begin{sublemma}
It holds that $\lambda(A \cap B(a)) \lesssim \min\left(
a^4, \, r_2^2a^3 / \rho, \, r_1r_2^2a, \, r_1^2r_2^2 \right)$.

\begin{proof}
The projection of $A \cap B(a)$ to the $yz$-plane is contained in
each of the parallelograms
	\[
	S_1 \cap S_2, \qquad\qquad
	S_3 \cap S_4, \qquad\qquad
	S_2 \cap S_4.
	\]
The first two have area $4a^3$ and $16 r_1r_2^2$, respectively,
and the third has vertices
	\[
	\left(\frac{-2r_2^2 \pm a^2}{2\rho}, -2r_2^2 \right),
	\qquad\qquad
	\left(\frac{2r_2^2 \pm a^2}{2\rho}, 2r_2^2 \right),
	\]
and hence area $4r_2^2a^2 / \rho$. The extension of $A$ in the $x$-direction
is $4r_1$ and the extension of $B(a)$ in the $x$-direction is $2a$.
Thus
	\begin{align*}
	\lambda(A \cap B(a)) &\lesssim
	\min\left(a^3, \, r_1r_2^2, \frac{r_2^2a^2}{\rho}\right) \cdot
	\min\left(r_1, a\right) \\
	&\leq
	\min\left(
	a^4, \, \frac{r_2^2a^3}{\rho}, \, r_1r_2^2a, \, r_1^2r_2^2
	\right). \qedhere
	\end{align*}

\end{proof}
\end{sublemma}

\begin{proof}[Proof of Lemma~\ref{energyestlemma} (continued). The case 
$r_1 \geq r_2$]
Again, $R_t(p)$ can be estimated using \eqref{Ptesteq},
the sublemma and~\eqref{intesteq}, but the estimate now depends
on $\rho\in [0,r_1]$, the first component of the point $p$. Let
$t \in (0, 4) \setminus\{1, 3\}$. Then for every $\rho$,
	\begin{equation}\label{allrho}
	\begin{split} 
	R_t(p) &\lesssim
	\int_0^\infty
	\min\left(
	a^4, \, \, r_1r_2^2a, \, r_1^2r_2^2
	\right) a^{-(t + 1)} \intd a \\
	&\leq
	\int_0^{\left(r_1r_2^2\right)^{1 / 3}} a^{3 - t} \intd a
	+
	\int_{\left(r_1r_2^2\right)^{1 / 3}}^{r_1}
	r_1r_2^2a^{-t} \intd a
	+
	\int_{r_1}^\infty r_1^2r_2^2 a^{-(t + 1)} \intd a \\
	&\lesssim
	\max\left(
	\left(r_1r_2^2\right)^{(4 - t) / 3}, \,
	r_1^{2 - t} r_2^2
	\right)
	=
	\begin{cases}
	r_1^{2 - t} r_2^2 & \text{if } t \in (0, 1),\\
	\left(r_1r_2^2\right)^{(4 - t) / 3} & \text{if } t \in (1, 4) \setminus
        \{3\}.
	\end{cases}
	\end{split}
	\end{equation}
For $t\in(1,4)\setminus\{3\}$ and $\rho \geq \left(r_2^4 / r_1\right)^{1 / 3}$,
the estimate can be made sharper. For such $\rho$, the interval
$[r_2^2 / \rho, \sqrt{\rho r_1}]$ is non-empty, and when $a$ lies in this
interval the minimum in the expression given by the sublemma is achieved by
the second option. Thus
	\begin{align*}
	R_t(p) &\lesssim
	\int_0^\infty
	\min\left(
	a^4, \, \frac{r_2^2a^3}{\rho}, \, r_1r_2^2a, \, r_1^2r_2^2
	\right) a^{-(t + 1)} \intd a \\
	&\leq
	\int_0^{r_2^2 / \rho} a^{3 - t} \intd a
	+
	\int_{r_2^2 / \rho}^{\sqrt{\rho r_1}}  \frac{r_2^2a^{2 - t}}{\rho} \intd a
	+
	\int_{\sqrt{\rho r_1}}^{r_1} r_1r_2^2a^{-t} \intd a
	+
	\int_{r_1}^{\infty} r_1^2r_2^2a^{-(t + 1)} \intd a \\
	&\lesssim
	\max\left(
	r_2^{2(4 - t)} \rho^{t - 4}, \, r_1^{(3 - t) / 2} r_2^2 \rho^{(1 - t) / 2}, \,
	r_1^{2 - t} r_2^2
	\right) \\
	&\stackrel{*}{=}
	\begin{cases}
	r_1^{(3 - t) / 2} r_2^2 \rho^{(1 - t) / 2} & \text{if } t \in (1, 3), \\
	r_2^{2(4 - t)} \rho^{t - 4} & \text{if } t \in (3, 4).
	\end{cases}
	\end{align*}
Denoting the options in the maximum by $O_1, O_2, O_3$,
	\[
	\frac{O_1^2}{O_2^2} = \left(\frac{r_1\rho^3}{r_2^4}\right)^{t - 3},
	\qquad\qquad
	\frac{O_2^2}{O_3^2} = \left(\frac{r_1}{\rho}\right)^{t - 1},
	\]
and the equality $*$ follows using that the expressions in parentheses
are greater than or equal to $1$.

Let $g_t(\rho) = \sup_{z_0} R_t(p)$ where $p = (\rho, 0, z_0)$.
Then
	\[
	I_t(R) \lesssim r_2^2 \int_0^{r_1} g_t(\rho) \rho \intd\rho.
	\]
For $t \in (0, 1)$, the estimate~\eqref{allrho} gives
	\[
	I_t(R) \lesssim
	r_2^2 \int_0^{r_1} r_1^{2 - t} r_2^2 \rho \intd\rho
	\approx
	r_1^{4 - t} r_2^4.
	\]
For $t \in (1, 4) \setminus \{3\}$, the first part of the integral
$\int_0^{r_1} g_t(\rho) \rho \intd\rho$ is again estimated
using~\eqref{allrho}. Let
	\[
	I_0 = r_2^2 \int_0^{\left(r_2^4 / r_1\right)^{1 / 3}}
	\left(r_1r_2^2\right)^{(4 - t) / 3}
	\rho\intd\rho
	\approx
	r_1^{(2 - t) / 3} r_2^{2(11 - t) / 3}.
	\]
Then for $t \in (1, 3)$,
	\[
	I_t(R) \lesssim
	I_0
	+
	r_2^2 \int_0^{r_1}
	r_1^{(3 - t) / 2} r_2^2 \rho^{(3 - t) / 2} \intd\rho
	\approx
	I_0 + r_1^{4 - t}r_2^4
	\approx
	r_1^{4 - t}r_2^4,
	\]
using that
$I_0 \approx r_1^{(2 - t) / 3} r_2^{2(5 - t) / 3} r_2^4 \leq r_1^{4 - t}r_2^4$.
For $t \in (3, 4)$,
	\[
	I_t(R) \lesssim
	I_0
	+
	r_2^2 \int_0^{r_1} r_2^{2(4 - t)} \rho^{t - 3} \intd\rho
	\approx
	I_0 + r_1^{t - 2} r_2^{2(5 - t)}
	\approx r_1^{t - 2} r_2^{2(5 - t)},
	\]
using that
$I_0 \approx r_1^{(2 - t) / 3} r_2^{4(t - 2) / 3} r_2^{2(5 - t)} \leq
r_1^{t - 2} r_2^{2(5 - t)}$.
This completes the proof of Lemma~\ref{energyestlemma}.
\end{proof}

Combining Lemmas~\ref{contentestlemma} and \ref{energyestlemma} 
proves Proposition~\ref{phiestimate}.
  
\section{Proof of Theorem~\ref{mainthm}}

It is clear that $\dimh E_{\seq r}(\omega) \leq \dimh \heis = 4$.
Let $t \in (0, 4) \setminus \{1, 2, 3\}$.

For every $\omega$ and $n_0$,
	\[
	E_{\seq r}(\omega) \subset \bigcup_{n = n_0}^\infty
	\crect{\omega_n}{r_n},
	\]
and by Lemma~\ref{contentestlemma},
$\hmeas_\infty^t(\crect{\omega_n}{r_n}) \lesssim \Phi^t(r_n)$. 
Thus for every $n_0$
	\[
	\hmeas_\infty^t(E_{\seq r}(\omega)) \lesssim
	\sum_{n = n_0}^{\infty} \Phi^t(r_n).
	\]
It follows that if $\sum_n \Phi^t(r_n) < \infty$ then
$\hmeas_\infty^t(E_{\seq r}(\omega)) = 0$ so that
$\dimh E_{\seq r}(\omega) \leq t$.

By Lemma~\ref{energyestlemma},
$\Cap_t(\crect{0}{r_n}) \gtrsim \Phi^t(r_n)$.
Thus if $\sum_n \Phi^t(r_n) = \infty$ then Theorem~\ref{auxthm}
with $\lambda$ replaced by $\lambda(W)^{-1}\lambda$ implies
that $\dimh E_{\seq r}(\omega) \geq t$ for almost every $\omega$.

\section*{Acknowledgements}
We gratefully acknowledge the support of the Centre of Excellence in Analysis
and Dynamics Research, funded by the Academy of Finland, and the hospitality of
Institut Mittag--Leffler where part of this work was carried out. We also thank
the referee for useful comments.


\begin{thebibliography}{28}

\bibitem{BBMT10}
Balogh, Z., Berger, R., Monti, R., and Tyson, J.,
\emph{Exceptional sets for self-similar fractals in Carnot groups},
Math. Proc. Cambridge Philos. Soc. 149 (2010), no. 1, 147--172.

\bibitem{BD-CFMT13}
Balogh, Z., Durand-Cartagena, E., F\"assler, K., Mattila, P., and Tyson, J.,
\emph{The effect of projections on dimension in the Heisenberg group},
Rev. Mat. Iberoam. 29 (2013), no. 2, 381--432.

\bibitem{BFMT12}
Balogh, Z., F\"assler, K., Mattila, P., and Tyson, J.,
\emph{Projection and slicing theorems in Heisenberg groups},
Adv. Math. 231 (2012), no. 2, 569--604.

\bibitem{BaloghRicklyCassano03}
Balogh, Z., Rickly, M., and Serra Cassano, F.,
\emph{Comparison of Hausdorff measures with respect to the Euclidean and the
Heisenberg metric}, Publ. Mat. 47 (2003), no. 1, 237--259.

\bibitem{BaloghTyson05}
Balogh, Z., and Tyson, J.,
\emph{Hausdorff dimensions of self-similar and self-affine fractals in the
Heisenberg group}, Proc. London Math. Soc. (3) 91 (2005), no. 1, 153--183.

\bibitem{BaloghTysonWarhurst09}
Balogh, Z., Tyson, J., and Warhurst, B.,
\emph{Sub-Riemannian vs. Euclidean dimension comparison and fractal geometry on
Carnot groups}, Adv. Math. 220 (2009), no. 2, 560--619.

\bibitem{BaloghTysonWildrick17}
Balogh, Z., Tyson, J., and Wildrick, K.,
\emph{Frequency of Sobolev dimension distortion of horizontal subgroups in
Heisenberg groups},
Ann. Sc. Norm. Super. Pisa Cl. Sci. (5) 17 (2017), no. 2, 655--683.

\bibitem{Besicovitch34}
Besicovitch, A.,
\emph{On the sum of digits of real numbers represented in the dyadic system},
Ann. Math. 110 (1935), no. 1, 321--330.

\bibitem{Borel97}
Borel, E.,
\emph{Sur les s\'eries de Taylor}, Acta Math. 21 (1897), no. 1, 243--247.

\bibitem{Cantelli17}
Cantelli, F.,
\emph{Sulla probabilit\'a come limite della frequenza},
Atti Accad. Naz. Lincei 26 (1917), no. 1, 39--45.

\bibitem{ChousionisTysonUrbanski19}
Chousionis, V., Tyson, J., and Urba\'nski, M.,
\emph{Conformal graph directed Markov systems on Carnot groups},
Mem. Amer. Math. Soc., to appear, available in arXiv:1605.01127. 

\bibitem{Durand10}
Durand, A.,
\emph{On randomly placed arcs on the circle},
in ``Recent developments in fractals and related fields'',
Appl. Numer. Harmon. Anal., Birkh\"auser Boston Inc., 2010, pp. 343--351.

\bibitem{Eggleston49}
Eggleston, H.,
\emph{The fractional dimension of a set defined by decimal properties},
Quart. J. Math. Oxford 20 (1949), 31--36.

\bibitem{EJJS18}
Ekstr\"om, F., J\"arvenp\"a\"a, E., J\"arvenp\"a\"a, M., and Suomala, V.,
\emph{Hausdorff dimension of limsup sets of random rectangles in products of
regular spaces},
Proc. Amer. Math. Soc. 146 (2018), no. 6, 2509--2521.

\bibitem{EkstromPerssonUP}
Ekstr\"om, F., and Persson, T.,
\emph{Hausdorff dimension of random limsup sets}, J. London Math. Soc. 98
(2018), no. 2, 661--686.

\bibitem{Falconer94}
Falconer, K.,
\emph{Sets with large intersection properties},
J. London Math. Soc. (2) 49 (1994), no. 2, 267--280.

\bibitem{FanSchmelingTroubetzkoy13}
Fan, A.-H., Schmeling, J., and Troubetzkoy, S.,
\emph{A multifractal mass transference principle for Gibbs measures with
applications to dynamical Diophantine approximation},
Proc. Lond. Math. Soc. 107 (2013), no. 5, 1173--1219.

\bibitem{FanWu04}
Fan, A.-H., and Wu, J.,
\emph{On the covering by small random intervals},
Ann. Inst. Henri Poincar\'e Probab. Stat. 40 (2004), no. 1, 125--131.

\bibitem{FasslerHovila16}
F\"assler, K., and Hovila, R.,
\emph{Improved Hausdorff dimension estimate for vertical projections in the
Heisenberg group},
Ann. Sc. Norm. Super. Pisa Cl. Sci. (5) 15 (2016), no. 1, 459--483.

\bibitem{FJJSUP}
Feng, D.-J., J\"arvenp\"a\"a, E., J\"arvenp\"a\"a, M., and Suomala, V.,
\emph{Dimensions of random covering sets in Riemann manifolds},
Ann. Probab. 46 (2018), no. 3, 1542--1596.

\bibitem{Hovila14}
Hovila, R.,
\emph{Transversality of isotropic projections, unrectifiability, and Heisenberg
groups},
Rev. Mat. Iberoam. 30 (2014), no. 2, 463--476.

\bibitem{Jarnik32}
Jarn\'\i k, V.,
\emph{Zur theorie der diophantischen approximationen},
Monatsh. Math. Phys. 39 (1932), no. 1, 403--438.

\bibitem{JJKLS14}
J\"arvenp\"a\"a, E., J\"arvenp\"a\"a, M., Koivusalo, H., Li, B., and
Suomala, V.,
\emph{Hausdorff dimension of affine random covering sets in torus},
Ann. Inst. Henri Poincar\'e Probab. Stat. 50 (2014), no. 4, 1371--1384.

\bibitem{JJKLSX17}
J\"arvenp\"a\"a, E., J\"arvenp\"a\"a, M., Koivusalo, H., Li, B., Suomala, V.,
and Xiao, Y.,
\emph{Hitting probabilities of random covering sets in tori and metric spaces},
Electron. J. Probab. 22 (2017), no. 1, 1--18.

\bibitem{Khintchine26}
Khintchine, A.,
\emph{Zur metrischen theorie der diophantischen approximationen},
Math. Z. 24 (1926), no. 1, 706--714.

\bibitem{Persson15}
Persson, T.,
\emph{A note on random coverings of tori},
Bull. Lond. Math. Soc. 47 (2015), no. 1, 7--12.

\bibitem{PerssonUP}
Persson, T.,
\emph{Inhomogeneous potentials, Hausdorff dimension and shrinking targets},
Ann. Henri Lebesgue, to appear.

\bibitem{PerssonReeve15}
Persson, T., and Reeve, H.,
\emph{A Frostman-type lemma for sets with large intersections, and an
application to Diophantine approximation},
Proc. Edinb. Math. Soc. (2) 58 (2015), no. 2, 521--542.

\bibitem{SeuretUP}
Seuret, S.,
\emph{Inhomogeneous random coverings of topological Markov shifts},
Math. Proc. Cambridge Philos. Soc. 165 (2018), no. 2, 341--357.

\bibitem{SeuretVigneron17}
Seuret, S., and Vigneron, F.,
\emph{Multifractal analysis of functions on Heisenberg and Carnot groups},
J. Inst. Math. Jussieu 16 (2017), no. 1, 1--38.

\bibitem{Vandehey16}
Vandehey, J.,
\emph{Diophantine properties of continued fractions on the Heisenberg group},
Int. J. Number Theory 12 (2016), no. 2, 541--560.

\bibitem{ZhengUP}
Zheng, C.,
\emph{A shrinking target problem with target at infinity in rank one
homogeneous spaces}, arXiv:1610.01870v3 (2016).  


\end{thebibliography}
\end{document}